\documentclass[12pt]{amsart}
\usepackage{amsmath}
\usepackage{latexsym}
\usepackage{amssymb}
\usepackage{graphicx}
\usepackage{tikz}

\usetikzlibrary{calc}
\usetikzlibrary{shadows}
\usetikzlibrary{shapes}

\DeclareMathOperator{\co}{co}

\DeclareMathOperator*{\argmin}{arg \ min}
\newtheorem{theorem}{Theorem}
\newtheorem{proposition}[theorem]{Proposition}
\newtheorem{lemma}[theorem]{Lemma}
\newtheorem{corollary}[theorem]{Corollary}
\newtheorem{definition}[theorem]{Definition}

\newtheorem{remark}[theorem]{Remark}

\usepackage{mathrsfs}
\begin{document}

%
\def\R {{\mathbb{R}}}
\def\N {{\mathbb{N}}}
\def\C {{\mathbb{C}}}
\def\Z {{\mathbb{Z}}}
\def\phi{\varphi}
\def\epsilon{\varepsilon}
\def\ma{{\mathcal A}}
%
\def\tb#1{\|\kern -1.2pt | #1 \|\kern -1.2pt |} 
\def\Qed{\qed\par\medskip\noindent}
%

\title[Singularities Riemannian distance]{Obstacles and Singularities\\ of Riemannian Distance Functions}  
\author{Paolo Albano} 
\address{Dipartimento di Matematica, Universit\`a di Bologna, Piazza di Porta San Donato 5, 40127 Bologna, Italy}
\email{paolo.albano@unibo.it}
 \author{Piermarco Cannarsa} 
\address{Dipartimento di Matematica, 
Universit\`a di Roma "Tor Vergata", Via della Ricerca Scientifica, 00133 Roma, Italy}
\email{cannarsa@mat.uniroma2.it}
\author[Carlo Sinestrari]{Carlo Sinestrari}
\address{Dipartimento di Matematica, 
Universit\`a di Roma "Tor Vergata", Via della Ricerca Scientifica, 00133 Roma, Italy}
\email{sinestra@mat.uniroma2.it}

\date{\today}

\begin{abstract}
 We study the distance function from a point target in the complement of a compact obstacle endowed with a smooth Riemannian metric. We prove that the obstacle necessarily generates singularities of the distance function: every sufficiently high level set contains a singular point. We also show that every singular point outside the obstacle belongs to a nontrivial Lipschitz arc of singularities, thereby extending to the constrained setting classical propagation results for Hamilton--Jacobi equations. Finally, we provide examples showing that these results are essentially sharp, including a nonconvex obstacle for which the distance function is differentiable at every boundary point.
 \end{abstract}


\maketitle

\section{Introduction and main results} 
In this paper we study the singularities of Riemannian distance functions in the presence of obstacles. More precisely, given a compact obstacle $\mathscr O\subset \R^n$ with smooth boundary and a point target $k_0\in \R^n\setminus\mathscr O$, where $n\geq 2$, we investigate the regularity of the distance function from $k_0$ in the closure of the complement of the obstacle. Distance functions play a central role in geometry, optimal control, and Hamilton--Jacobi theory. Their singularities encode the structure of minimizing trajectories and can be regarded as a constrained counterpart of the classical cut locus in Riemannian geometry. 

In the unconstrained setting, the relation between singularities and nonuniqueness of minimizing geodesics is well understood, and the propagation of singularities has been extensively studied through the theory of semiconcave functions and viscosity solutions of Hamilton--Jacobi equations, see e.g. \cite{CS}. Much less is known when state constraints are present. The presence of an obstacle fundamentally modifies the geometry of minimizing trajectories. Even in the Euclidean case, minimizers may touch the boundary of the obstacle, lose smoothness, and exhibit branching phenomena that have no analogue in the unconstrained setting. A natural question is therefore whether the obstacle necessarily creates singularities of the distance function and, if so, how these singularities are organized. 

The present paper continues and completes the analysis initiated in \cite{ABC1}. Besides extending to general Riemannian metrics several results established in \cite{ABC1} for the Euclidean distance, our main goal is to understand the global structure of the singular set generated by the obstacle. In particular, we address the following questions: \begin{itemize} \item Does the presence of a bounded obstacle necessarily create singularities of the distance function? \item Must the boundary of the obstacle contain singular points? \item Do singularities propagate in the presence of state constraints as they do for semiconcave solutions of Hamilton--Jacobi equations? \end{itemize} As we shall see, the answers reveal several phenomena that are specific to the constrained setting. 

To be definite, let us introduce the class $\mathcal{O}$ of the obstacles we deal with. We say that a set $\mathscr{O}$ belongs to the class $\mathcal{O}$ if it is the closure of a bounded open set with
$C^2$ boundary. We denote by \[ X:=\overline{\R^n\setminus \mathscr{O}} \] the closure of the complement of the obstacle and we assume, without loss of generality, that $X$ is arcwise connected. In order to define a metric on $X$, we consider a map taking values in the set of symmetric $n\times n$ matrices, $\R^n\ni x\mapsto A(x)$, such that
 \begin{equation}\label{eq:A} \begin{cases} x\mapsto A(x) \text{ is of class $C^2$}, \\ A(x)\geq I\, c_0^2 \quad \text{as quadratic forms}, \end{cases} \end{equation} 
 for a suitable $c_0>0$ and for every $x\in \R^n$, where $I$ is the identity matrix. Given a curve $\gamma\in C^1([0,1];X)$, we define 
  \[ L(\gamma )= \int_0^1 \langle A(\gamma (t)) \dot{\gamma}(t), \dot{\gamma}(t)\rangle^{1/2} \, dt , \] 
   where $\langle \cdot , \cdot \rangle$ stands for the standard Euclidean scalar product. For every $x,y\in X$, the distance between $x$ and $y$ is given by 
    \begin{equation} \label{def.distance} \mbox{\rm dist}(x,y)=\inf_{\gamma} L(\gamma ), \end{equation} 
     where the infimum is taken over all curves $\gamma\in C^1([0,1];X)$ such that $\gamma (0)=x$ and $\gamma (1)=y$. We then fix an arbitrary point $k_0\in \R^n \setminus \mathscr{O}$ and define, for $x \in X$, 
      \begin{equation} \label{d} d(x)=\mbox{\rm dist}(x,k_0), \end{equation} 
       the distance function from $k_0$. Such a function can be seen as the value function for a constrained minimum time problem with point target $k_0$. Let us also recall that $d$ is the viscosity solution of a suitable boundary value problem for the eikonal equation, see e.g. \cite[Theorem X.1]{CDL}, %
    \begin{equation}\label{eq:iconale} \langle A^{-1} Dd(x),Dd(x)\rangle =1 \qquad \text{in } X\setminus (\partial \mathscr{O}\cup \{ k_0\}). \end{equation} 
     Notice that, due to \eqref{eq:A}, 
      \[ d(x,y)\geq c_0|x-y|,\qquad d(x)\geq c_0|x-k_0|, \qquad \forall x,y\in X, \] 
       where $|\cdot |$ is the Euclidean norm. The last inequality implies that $d$ is coercive, i.e. $\lim_{|x|\to \infty}d(x)=\infty$. It is also easy to see that $d$ is locally Lipschitz continuous on $X$. The regularity of $d$ will be described in more detail in Section 2. 
       
       The problem of minimizing distance in the presence of an obstacle, possibly in infinite-dimensional spaces, has been investigated by several authors, see for instance \cite{AA,Al,ABB,Ca,MS,RS}. Most of the existing literature is concerned with the existence, uniqueness, and regularity of minimizing geodesics. Our perspective is different. 
       
       We focus on the distance function itself and on the relation between its differential properties and the behaviour of minimizing trajectories, in the spirit of dynamic programming and Hamilton--Jacobi theory. 
        More specifically, we study the singular set of the distance function, denoted by $\Sigma(d)$. For interior points of $X$, singularity simply means lack of differentiability of $d$; for boundary points, a slightly modified definition is required, see Definition \ref{insiemes} and Corollary \ref{r:ext}. Since the target point $k_0$ is always singular, independently of the metric and of the obstacle, we restrict our attention to points different from $k_0$. 
        
        When no obstacle is present, it is well known that $x\neq k_0$ is singular if and only if there exists more than one length-minimizing arc between $x$ and $k_0$. Consequently, for some metrics---for instance the Euclidean metric, corresponding to $A(x)\equiv I$---the target is the only singular point, whereas for more general metrics singularities may arise through the crossing of geodesics. 
        
        In the presence of an obstacle the situation becomes substantially richer.  We prove that a point $x$ is singular if and only if there exist at least two minimizing trajectories with different initial velocities from $x$ to $k_0$. At the same time, distinct minimizers may have the same initial velocity and branch only after reaching the boundary of the obstacle, a phenomenon that cannot occur in the unconstrained case. 
       
       Our first main result shows that singularities are not exceptional. In fact, every sufficiently high level set of the distance function necessarily contains singular points. Thus, the presence of a bounded obstacle inevitably creates singularities, regardless of the choice of the Riemannian metric. \begin{theorem}\label{t:existence} Let $d$ be defined by \eqref{d}, for some metric $A(\cdot)$ satisfying \eqref{eq:A}, an obstacle $\mathscr{O}\in \mathcal{O}$ and a target $k_0 \in \R^n \setminus \mathscr{O}$. Let us set \begin{equation} \label{dplus} d_+:=\max_{x\in\partial\mathscr O} d(x). \end{equation} Then, for any $a>d_+$, any connected component of the level set \[ \{ x\in X\mid d(x)=a\} \] has nonempty intersection with the singular set $\Sigma(d)$. \end{theorem} It is natural to wonder whether the boundary of the obstacle should also contain at least one singular point for the distance. In the case of the Euclidean distance and of a convex obstacle, Theorem 1.6 in \cite{ABC2} shows that the point on $\partial\mathscr O$ where the maximum in \eqref{dplus} is attained is singular for $d$. We show instead that this property may fail if the obstacle is not convex. More precisely, in Section \ref{whale} we give an example of an obstacle ${\mathscr O} \subset \R^2$ such that the associated distance from a suitable target $k_0$, with the Euclidean metric, is regular at every point of $\partial \mathscr{O}$. 
       
       Once the existence of singularities of $d$ in $X\setminus \{k_0\}$ is established, one may ask whether such singularities can be isolated. Our second main result gives a negative answer for all singular points outside the obstacle. \begin{theorem}\label{t:propagation} Let $d$ be defined by \eqref{d}, under the same assumptions as in Theorem \ref{t:existence}. Then, for every $x\in \Sigma (d)\cap (\R^n\setminus \mathscr{O})$, with $x\neq k_0$, there exist $\sigma >0$ and a Lipschitz continuous arc ${\bf x}:[0,\sigma [\longrightarrow \Sigma (d)$ such that ${\bf x}(0)=x$ and ${\bf x}(t)\neq x$, for every $t\in ]0,\sigma [$. \end{theorem} In the Euclidean case, a proof of Theorem \ref{t:propagation} is given in \cite{ABC2}. In addition to extending the result to the Riemannian case, we also give a counterexample to global propagation, by showing that the singular arc may be bounded and ``disappear'' when hitting the obstacle. 
       
       The paper is organized as follows. In Section 2 we recall some properties of semiconcave functions. In Section 3 we study distance-minimizing arcs in the presence of an obstacle, with particular emphasis on their relation with the classical and generalized gradients of $d$. Sections 4 and 6 are devoted to the proofs of Theorem \ref{t:existence} and Theorem \ref{t:propagation}, respectively. Section 5 contains the example of a connected obstacle $\mathscr O \subset \R^2$ for which the associated distance is differentiable at every point of $\partial \mathscr O$. Finally, Section 7 gives some finer properties of the singular set in the two-dimensional Euclidean case.

\section{Preliminaries} 
\setcounter{equation}{0}
\setcounter{theorem}{0}

It is well known that the distance function is more regular than merely Lipschitz continuous.  As shown in \cite{ABC2}, the appropriate regularity class for $d$ in the presence of an obstacle is the one of semiconcave functions with fractional modulus. Here and in the following, we write $[x,y]$ to denote the line segment between two points $x,y \in \R^n$.

\begin{definition}
Given $U\subset \R^n$, we say that a function $u:U\longrightarrow \R$ is {\em semiconcave with fractional modulus} of exponent  $\alpha \in ]0,1]$ 
if $u$ is locally Lipschitz continuous on $U$
and   
there exists  $C\in \R$ such that 
\begin{equation}\label{eq:sca}
\lambda u(x)+(1-\lambda) u(y)-u(\lambda x + (1-\lambda )y)\le C\lambda (1-\lambda) |x-y|^{1+\alpha},
\end{equation}
for any $x,y\in U$ such that $[x,y] \subset U $ 
 and for every $\lambda\in [0,1]$. 
\end{definition}
In the case $\alpha =1$ we say that $u$ is {\em semiconcave with linear modulus}. We denote by $SC^\alpha (U)$ the set of all the fractionally semiconcave functions of exponent $\alpha$ in $U$. Such a property can obviously be made local, in which case we refer to the space $SC_{loc}^\alpha (U)$. The set $U$ for the moment is arbitrary, but later in this section we will restrict ourselves to the case where $U$ is the closure of an open set.

The next statement describes the semiconcavity properties of the distance function with an obstacle. 

\begin{theorem} \label{t:rb}
Assume \eqref{eq:A} and let $\mathscr{O}\in \mathcal{O}$. Then the distance function $d$ satisfies \medskip \\
{\rm (i) (Interior regularity)}   $d\in SC_{loc}^1( X \setminus (\partial \mathscr{O}\cup \{ k_0\} ))$;  \\
{\rm (ii) (Boundary Regularity) } $d\in SC^{\frac 12}_{loc}(X\setminus  \{ k_0\})$.  
\end{theorem} 
\begin{proof}
Part (i) is a well-known property of solutions of eikonal-type equation, see \cite{L} and \cite{A3}, while (ii) is proved in \cite{ABC2}.
\end{proof}

The regularity in Theorem \ref{t:rb}(ii) is optimal. Indeed, it is shown in \cite{A6} that, already in the Euclidean case, $d$ does not belong to $SC^\alpha_{loc}(X\setminus \{ k_0\})$ for any $\alpha \in ]1/2,1]$. 

We now recall some definitions of nonsmooth analysis.

\begin{definition}\label{dstar}(Generalized differentials) Let $u:U \subset \R^n \to \R$ be a locally Lipschitz continuous function. The set $D^*u(x)$ of {\em reachable gradients} of $u$ at $x \in U$ is defined as 
\begin{multline}\label{eq:dstar}
D^*u(x)=\{ p\in \R^n\ | \  \exists x_h\in \overset{\circ}{U},  x_h\to x, \  \exists Dd(x_h)\to p\} .
\end{multline}
The {\em (Fr\'echet) superdifferential} of $u$ at a point $x \in U$ is
defined as 
\begin{multline}\label{eq:superdiff}
D^+u(x)=\left\{ p\in \R^n\ | \ \limsup_{y \to x} \frac{u(y)-u(x)-\langle p, y-x \rangle}{|y-x|} \leq 0 \right\}.
\end{multline}
\end{definition}

By Rademacher's theorem, if $u$ is locally Lipschitz then $D^*u(x)$ is nonempty and compact at every point 
$x \in \mathring{U}$; it is also nonempty at points $x \in \partial U$, if $\partial U$ is sufficiently regular. If $u$ is semiconcave, $D^+u(x)$ is also nonempty and satisfies the following properties, see \cite{CS}.

\begin{theorem} \label{t:generdiff}
Let $u \in SC^\alpha_{loc}(U)$ for some $\alpha \in \,]0,1]$ and $U \subset \R^n$. Then, for any $x \in \mathring{U}$ we have 
\begin{enumerate} 
\item
 $D^+u(x)$ is nonempty and coincides with the convex hull of $D^*u(x)$. \\
\item If $x_k \to x$ and $p_k \to p$, with $p_k \in D^+u(x_k)$, then $p \in D^+u(x)$. \\
\item $u$ is differentiable at $x$ if and only if $D^+u(x)$ is a singleton. \\
\item for every compact set $K\subset  U$ containing $x$, for every $y\in K$ such that $[x,y]\subset K$, and every $p\in D^+d(x)$, we have that   
\begin{equation}\label{p}
u(y)\le u(x)+\langle p,y-x\rangle +C |y-x|^{1+\alpha}, 
\end{equation}
for a suitable constant $C$ depending on $K$. 
\end{enumerate}
 \end{theorem}
 
If $x \in \partial U$, the superdifferential $D^+u(x)$ is not well behaved and some of the above properties may fail. In Remark \ref{r:int} below we give partial generalization of \eqref{p} to boundary points.
 
We now recall the result about the extension of a semiconcave function
with fractional modulus, obtained in \cite{ABC1}, and which will be used several times in the sequel. 
%

 \begin{theorem}\label{t:estensione}
Let $U\subset \R^n$ be the closure of an open set, and let $u\in SC_{loc}^\alpha (U)$. Then, for every $x\in \partial U$, there exist $\delta>0$ and a function $E(u)\in SC^\alpha (B_\delta (x))$ such that 
\begin{enumerate} 
\item $E(u)(y)=u(y)$ for every $y\in B_\delta (x)\cap U$;
\item $D^*E(u)(y)=D^*u(y)$ for every $y\in B_\delta (x)\cap \partial U$.
\end{enumerate}
\end{theorem}
In particular, for our applications, we will take $u=d$, $\alpha =1/2$ and $U=X \cap \overline{B}_r(k_0)$, for a suitable positive $r$ less than the (Euclidean) distance of $k_0$ from $\partial\mathscr {O}$.

 \begin{remark}\label{r:int} 
 By applying the above result, we see that
  inequality \eqref{p} also holds when $x \in \partial U \cap K$ and $p \in D^*u(x)$.
   \end{remark}

We now introduce the singular set for a semiconcave function in a domain with boundary.

\begin{definition} \label{insiemes} (Singular set) Let $u \in SC^\alpha_{loc}(U)$ for some $\alpha \in \,]0,1]$ and $U \subset \R^n$. The {\em singular set} of $u$ is defined  as follows
$$
\Sigma (u)=\{ x\in U \ :
D^*u(x) \text{ has at least two elements}\}. 
$$
\end{definition}

Usually, e.g. \cite{CS}, the singular set of a semiconcave function is defined as the set of points of nondifferentiability. The relation with the above definition is explained in the next result.

\begin{corollary}\label{r:ext}
Let $u\in SC_{loc}^\alpha (U)$, with $U$ the closure of an open set.
\begin{enumerate} 
\item For $ x \in\mathring{U}$, we have that $ x \in \Sigma(u)$ if and only if $u$ is not differentiable at $x$.
\item For $ x \in \partial U$, we have that $ x \in \Sigma(u)$ if and only if each local semiconcave extension of $u$ is not differentiable at $x$. As a consequence, if $ x \notin \Sigma(u)$ then $u$ is differentiable at $x$.
\end{enumerate}
\end{corollary}
\begin{proof} Part (1) follows from (1) and (3) of Theorem \ref{t:generdiff}. Suppose now $x \in \partial U$. If $D^*u(x)$ contains $p_0,p_1$, with $p_0\not= p_1$, then, for every $\tilde{u}$ local extension of $u$, we have that $p_0,p_1$ belong to $D^*\tilde{u}(x)$ as well, by the definition of reachable gradients. Therefore $\tilde u$ is not differentiable at $x$ by part (1) of the statement.  Vice versa, if $x\notin \Sigma (u)$, then  Theorem~\ref{t:estensione} provides a local extension $\tilde u= E(u)$ such that $D^*\tilde u(x)$ is a singleton, hence it is differentiable at $x$. If the extension is differentiable at $x$, the original function $u$ also is by definition.
\end{proof} 

Simple examples show that the equivalence in (1) may fail at boundary points. For instance, let $U=\R^2 \setminus D$, where $D \subset \R^2$ is the open ball of radius $1$ centered at $(-1,0)$, and let $u(x,y)=-|x|$.  The function $u$ is concave, hence also semiconcave with a linear modulus. We have $D^*u(0,0)=\{(-1,0),(1,0)\}$, and so the point $(0,0)$ is singular. On the other hand, it is easy to check that
$$
\lim_{U \ni (x,y) \to (0,0)} \frac{u(x,y)-u(0,0)+x}{\sqrt{x^2+y^2}}=0,
$$
hence $u$ is differentiable at $(0,0)$ with $Du=(-1,0)$. On the other hand, any semiconcave extension of $u$ in a neighbourhood of $(0,0)$ would no longer be differentiable.

We conclude this section by recalling a general criterion for the ``propagation'' of singularities for semiconcave functions, see \cite{AC1}. 
\begin{theorem}\label{t:acprop}
Let  $U \subset \R^n$ be an open set, let  $u\in SC^1_{loc} (U)$, and let $x_0\in  \Sigma (u)$ be such that 
\begin{equation}\label{eq:h}
\partial D^+u(x_0)\setminus D^*u(x_0)\not=\emptyset .
\end{equation}
Then, for any $p_0 \in \partial D^+u(x_0)\setminus D^*u(x_0)$ and
 any $q_0\neq 0$ in the normal cone to $D^+u(x_0)$ at $p_0$, 
there is a Lipschitz continuous map 
$$
[0,\sigma [\ni s\mapsto {\bf x}(s)\in \Sigma (u)
$$
for a suitable $\sigma>0$, 
such that 
$$
{\bf x}(s)\not=x_0={\bf x}(0), \text{ for every  }s\in ]0,\sigma [, 
$$
and such that ${\bf \dot x}^+(0)=-q_0$, where ${\bf \dot x}^+(0)$ denotes the right derivative.
\end{theorem}
\section{Properties of the length minimizers}
\setcounter{equation}{0}
\setcounter{theorem}{0}

A curve $\gamma$ which achieves the infimum in \eqref{def.distance} for given $x,y \in X$ is called a {\em length minimizer}. 
We recall the following result about the existence and regularity of minimizers. In the following, we denote by $L(x,q)$ the Lagrangian associated with our metric
$$
L(x,q):=\langle A(x)q,q \rangle, \quad x \in \R^n, q \in R^n,
$$
and by $L_x$ and $L_q$ the partial derivatives with respect to $x,q$. 
\begin{theorem}\label{t:regmin}
Assume \eqref{eq:A} and let $\mathscr{O}\in \mathcal{O}$. Then, for every $x,y \in X$, there exists a $C^1$ curve $\gamma :[0,1]\rightarrow X$ which is a length minimizer. Moreover, any such minimizer satisfies \\
{\rm (i)}  $\dot{\gamma}$ is Lipschitz continuous on $[0,1]$, with uniform Lipschitz constant on all minimizers between endpoints $x,y$ varying in compact sets of $X$.   \\
{\rm (ii)} At any $t\in [0,1]$ such that  $\gamma (t)\notin \partial \mathscr{O}$ the curve $\gamma$ is twice differentiable and satisfies the geodesic ODE
\begin{equation}\label{geodesic}
\frac{d}{dt} L_q(\gamma(t),\dot{\gamma}(t))=L_x (\gamma(t),\dot{\gamma}(t)).
\end{equation}
{\rm (iii)} At any $t\in ]0,1[$ such that  $\gamma (t)\in \partial \mathscr{O}$ we have that $\dot{\gamma}(t)$ is a tangent vector to $\partial \mathscr{O}$ at $\gamma (t)$. 
\end{theorem} 

A proof of Theorem \ref{t:regmin} can be found in \cite[Lemma A.1]{RS} (see also \cite{AA}). Property (iii) is not explicitly given in the statement, but is obtained during the proof.

%
We now consider our fixed target point $k_0 \in X$ and, for every $x\in X$, we denote by $\Gamma^* [x]$  the set of all the length minimizers, joining $x$ with $k_0$, parametrized by the (Riemannian) arc length. That is, if $\gamma \in \Gamma^* [x]$, then it is parametrized by $t \in [0,d(x)]$ and satisfies
\begin{equation}\label{arcl}
\langle A(\gamma(t)) \dot \gamma(t), \dot \gamma(t) \rangle =1, \qquad t \in [0,d(x)].
\end{equation}
Here and in the sequel we mean that $\dot{\gamma}(0)$ is the right derivative while $\dot{\gamma }(d(x))$ stands for the left derivative, i.e.
$$\dot{\gamma}(0)=\lim_{t\to 0^+}\frac{\gamma (t)-\gamma (0)}t, \qquad \dot{\gamma}(d(x))=\lim_{t\to d(x)^-}\frac{\gamma (d(x))-\gamma (t)}{d(x)-t}.
$$ 

The geodesic equation \eqref{geodesic} also holds with this parametrization. In particular, it implies that,  if $\gamma \in \Gamma^* [x]$ and $[t_0,t_1] \subset [0,d(x)]$ is such that $\gamma(t) \notin \partial \mathscr O$ for all $t \in ]t_0,t_1[\,$, then $\gamma$ is uniquely determined by the values of $\gamma(t^*)$ and $\dot \gamma(t^*)$ for any fixed $t^* \in [t_0,t_1]$. In an interval where $\gamma$ touches $\partial \mathscr O$, instead, it can no longer be characterized as the solution of an ODE.
%

%
%
The next results shows that any point lying along a minimizer, different from the endpoints, is regular for the distance function, even on the boundary of the obstacle. In addition, we relate the elements of $\Gamma^* [x]$ with the elements of $D^*d(x)$. The corresponding result in the Euclidean case was obtained in \cite[Lemma 2.1]{ABC2}.
\begin{proposition}\label{l:vi}
Assume \eqref{eq:A}, let $\mathscr{O}\in \mathcal{O}$ and let $d$ be given by \eqref{d}. Then, for every $x\in X\setminus \{ k_0\}$ and  $\gamma \in \Gamma^*[x]$, we have that  
\begin{equation}\label{eq:s}
\dot{\gamma}(t)\in  -A^{-1}(\gamma (t)) D^*d(\gamma (t)),\qquad \forall t\in [0,d(x)].
\end{equation}
%

Furthermore,  for every $t\in ]0,d(x)[$, $\gamma (t)\notin \Sigma (d)$, i.e. 
\begin{equation}\label{eq:rlt}
D^*d(\gamma (t))=\{ Dd(\gamma (t))\},\quad \forall t\in ]0,d(x)[. 
\end{equation} 
Finally, for every $x\in X\setminus \{ k_0\}$ and for every $p\in D^*d(x)$ there exists $\gamma \in \Gamma^*[x]$ such that  $\dot{\gamma}(0)=-A^{-1}(x)p$.
\end{proposition} 

\begin{proof}[Proof of Proposition \ref{l:vi}] 
We  show first that \eqref{eq:s} holds.
 For every $\gamma \in \Gamma^*[x]$, we have that 
 \begin{equation}\label{dereq}
 d(x)-t=d(\gamma (t)),\qquad \text{for every}\qquad t\in [0,d(x)].
\end{equation}
  We claim that, for every $t\in [0,d(x)[\,$, we have
\begin{equation}\label{eq:divdir}
-1=\min_{q\in \co D^*d(\gamma (t) )}\langle q, \dot{\gamma} (t)\rangle= \min_{q\in  D^*d(\gamma (t))}\langle q,\dot{\gamma} (t)\rangle. 
\end{equation}
%
For all $t\in [0,d(x)[$ such that $\gamma (t)\notin \mathscr{O}$, \eqref{eq:divdir} follows from \cite[Theorem~3.3.6]{CS}   and the fact that  $\gamma$ is  $C^1$. If instead $\gamma (t)\in \partial \mathscr{O}$, then \eqref{eq:divdir} follows by applying the same argument to the  extension of $d$, $E(d)$, given by Theorem \ref{t:estensione} and recalling that $D^*E(d)(\gamma (t))=D^*d(\gamma (t))$. 


For simplicity of notation, we set in the following $p_t:=-\dot \gamma(t)$ and $A_t:=A(\gamma(t))$. We recall that, by \eqref{arcl}, we have
\begin{equation}\label{unit}
\langle A_t p_t, p_t \rangle =1.
\end{equation}
In addition, since $d$ is a viscosity solution of the eikonal equation \eqref{eq:iconale}, we find that 
\begin{equation}\label{subs}
	\langle A_t^{-1} q, q \rangle = 1, \quad \forall q \in D^*d(\gamma(t)).  
\end{equation}
Then from \eqref{eq:divdir}, \eqref{unit} and \eqref{subs} we deduce that there exists $q_t \in D^*d(\gamma(t))$ such that
\begin{equation}\label{chain}
 \langle q_t , p_t \rangle = - \langle q_t ,\dot \gamma(t) \rangle = 1 = 
 \frac 12 \langle A_t p_t, p_t \rangle + \frac 12 \langle A_t^{-1}q_t, q_t \rangle.
\end{equation}
It is well known that, for any positive definite $n \times n$ matrix $A$ and $p,q \in \R^n$ we have
\begin{equation}\label{fenchel}
\langle p,q \rangle \leq \frac 12 \langle Ap, p \rangle + \frac 12 \langle A^{-1}q, q \rangle,
\end{equation}
with equality if and only if $p=A^{-1}q$. Since by \eqref{chain} $p_t,q_t$ achieve the equality, we conclude that $\dot \gamma(t)=-p_t= -A_t^{-1}q_t$ for some  $q_t \in D^*d(\gamma(t))$.
Thus,  \eqref{eq:s} holds for every $t\in [0,d(x)[$.


The previous argument does not apply to the case $t=d(x)$ because
$\gamma (d(x))=k_0$ and $d(\cdot )$ is not semiconcave in any neighborhood of $k_0$.
To treat this case we observe that, by the previous step, we have 
$$
\dot{\gamma}(t)=-A^{-1}(\gamma (t))q_t,\qquad \text{for every }\; t\in [0,d(x)[ 
$$
with $q_t \in D^*d(\gamma(t))$. Using the fact that $\gamma \in C^{1,1}$ we can take the limit, as $t\to d(x)^-$, to find that $q_t$ converges to some $q\in D^*d(k_0)$ such that 
$$
-A(k_0)\dot{\gamma} (d(x))=q\in D^*d(k_0). 
$$ 
This completes the proof of \eqref{eq:s}. 

Now, let us show that $d$ is differentiable along a length minimizer, i.e. \eqref{eq:rlt} holds true. 
Let $x\in X\setminus \{ k_0\}$, $\gamma \in \Gamma^*[x]$, and $t\in ]0,d(x)[$.   By Theorem \ref{t:rb} and property \eqref{p}, see also Remark \ref{r:int}, we have that 
$$
d(\gamma (t-h))\le d(\gamma (t))+\langle q, \gamma (t-h)-\gamma (t)\rangle +C |\gamma (t-h)-\gamma (t)|^{\frac 32} ,
$$
for every $q\in D^*d(\gamma (t))$ and for every $h>0$ suitably small.  
We want to show that $D^*d(\gamma (t))$ is a singleton. We recall that 
$$
d(\gamma (t))=d(x)-t\qquad\text{and}\qquad d(\gamma (t-h))=d(x)-t+h. 
$$
Then, we find 
$$
h\le \langle q, \gamma (t-h)-\gamma (t)\rangle +C |\gamma (t-h)-\gamma (t)|^{\frac 32}\,.
$$
So, dividing both the sides of the inequality above by $h$ and taking the limit as $h\to 0^+$, we conclude that 
\begin{equation} \label{cond}
1\le \langle q, -\dot{\gamma} (t)\rangle ,
\end{equation}
for every $q \in D^*d(\gamma (t))$. 
%
On the other hand, recalling \eqref{fenchel}, \eqref{arcl} and \eqref{subs}, we find
$$
\langle q, -\dot{\gamma} (t)\rangle \leq \frac 12 \langle A^{-1}(\gamma (t))q,q\rangle +  \frac 12 
\langle A(\gamma (t))  \dot{\gamma}(t),   \dot{\gamma}(t)\rangle = 1.
$$
In view of \eqref{cond}, we conclude that $p=\dot{\gamma }(t)$ and $q$ achieve the equality in \eqref{fenchel}, which implies that 
$\dot{\gamma }(t)=-   A^{-1}(\gamma (t)) q$. Since $q \in D^*u(\gamma(t))$ is arbitrary, we conclude
$$
D^*d(\gamma (t))=\{ - A(\gamma (t))\dot{\gamma}(t)\}.
$$
This implies that $d$ is differentiable along $\gamma$ and
$$
\dot{\gamma}(t)=-A^{-1}(\gamma (t))Dd(\gamma (t)),
$$
for every $t\in ]0,d(x)[$, by possibly appealing to the extension of $d$ given by Theorem \ref{t:estensione} as above. 
Notice that the endpoints are always excluded.

It remains to prove that every element in $-A^{-1}(x) D^*d(x)$ can be taken as the initial velocity of a length minimizer. If $D^*d(x)$ is a singleton, this is an immediate consequence of the existence of minimizers and of \eqref{eq:s}. If $x$ is singular, let $p\in D^*d(x)$ and let $x_h$ be a sequence of regular points converging to $x$ such that $Dd(x_h)\to p$ as $h\to \infty$. By the previous step, there exists $\gamma_h\in \Gamma^*[x_h]$ such that $\dot{\gamma}_h(0)=-A^{-1}(x_h)Dd(x_h)$. By Theorem \ref{t:regmin}, both $\gamma_h$ and their derivatives are uniformly Lipschitz continuous. Hence, possibly taking a subsequence, we have that $\gamma_h$ uniformly converges to a limit $\gamma \in \Gamma ^*[x]$ and in addition
\begin{equation}\label{eq:crlt}
\lim_{h\to\infty} \| \dot{\gamma}_h-\dot{\gamma} \|_\infty =0. 
\end{equation}
 Now, let $s_h\in ]0,\min\{ d(x),d(x_h)\}]$ be a decreasing sequence converging to $0$ such that  
\begin{enumerate} 
\item   $|\dot{\gamma}_h (s_h)- \dot{\gamma} (s_h)|<\frac 1h$ (this can be achieved because of \eqref{eq:crlt});
\item $|\dot{\gamma}_h(s_h)-\dot{\gamma}_h(0)|<\frac 1h$ (here we use the fact that $\Gamma^*\subset C^{1,1}$),   
\end{enumerate}  
for every $h\in \N$.  Then, using once more the inclusion $\Gamma^* \subset C^{1,1}$, (1) and (2) above, 
 we conclude  that 
\begin{multline*}
\dot{\gamma}(0)=\lim_{h\to \infty } \dot{\gamma}(s_h)=\lim_{h\to \infty } \dot{\gamma}_h(s_h)=\lim_{h\to \infty }\dot{\gamma}_h (0)
\\
=-\lim_{h\to \infty }A^{-1}(x_h)Dd(x_h)=-A^{-1}(x)p.
\end{multline*}
This completes our proof of Lemma \ref{l:vi} 
\end{proof} 

 For the distance function without the obstacle, it is well-known that for any given $x$ and $p \in D^*(x)$, the trajectory $\gamma \in \Gamma^*[x]$ such that $\dot \gamma(0)=-A^{-1}(x)p$ is unique. A further property is that two distinct minimizing trajectories (either for the same point or for two distinct points) cannot intersect except possibly at the endpoints. In the presence of an obstacle, both properties can fail, as shown in Example 1 of papers \cite{Al} and \cite{ABB}. However, the properties still hold for trajectories before they touch the obstacle.

 \begin{proposition}\label{p:nuova}
(i) Let $x \in X \setminus \{k_0\}$ and let $\gamma_1 \in \Gamma^*[x]$. Let $t^* \in (0,d(x))$ be such that $\gamma(t) \notin \mathscr{O}$ for $t \in [0,t^*)$. Then, if there exists $\gamma_2 \in \Gamma^*[x]$ such that $\dot \gamma_2(0)=\dot \gamma_1(0)$, we have $\gamma_2(t)=\gamma_1(t)$ for all $t \in [0,t^*]$. \\
 (ii) Let $x_1,x_2 \in X \setminus \{k_0\}$, with $x_1 \neq x_2$ and let $\gamma_i \in \Gamma^*[x_i]$ for $i=1,2$. Suppose that there is $t^* \in (0,\min\{d(x_1),d(x_2)\} )$ such that $\gamma_i(t) \notin \mathscr{O}$ for $t \in [0,t^*)$. Then  $\gamma_1(t) \neq \gamma_2(t)$ for all $t \in (0,t^*]$.
  \end{proposition}
 \begin{proof}
 (i)  By hypothesis $x=\gamma_1(0)=\gamma_2(0) \notin \mathscr{O}$.
 Let $t_2 \in (0,d(x_2)]$ be the supremum of the $t$ such that $\gamma_2(t)  \notin \mathscr{O}$. Then  $\gamma_1$ and $\gamma_2$ are solutions of the geodesic ODE \eqref{geodesic} on the intervals $[0,t^*]$ and $[0,t_2]$ respectively, with the same initial conditions for the position and the velocity. Hence they coincide on the intersection of these intervals. This shows that $t_2 \geq t^*$ and that the curves  coincide on $[0,t^*]$. \\
 (ii)
 Suppose by contradiction that $\gamma_1(\bar t) = \gamma_2(\bar t)=:\bar x$ for some $\bar t \in (0,t^*]$. Then the previous Proposition implies that $d$ is differentiable at $\bar x$ and that
 $$
 \dot \gamma_1(\bar t)=\dot \gamma_2(\bar t)=-A^{-1}(\bar x)Dd(\bar x).
 $$
 In addition, both trajectories are solutions of the geodesic ODE on $[0,\bar t]$. Hence, they must coincide for $t \in [0,\bar t]$, in contradiction with the assumption $x_1 \neq x_2$.
 \end{proof}
 
Now, we are ready to prove our main results. 

\section{Proof of Theorem \ref{t:existence}} 
\setcounter{equation}{0}
\setcounter{theorem}{0}
 
We prove a preliminary result on the regularity of the level sets of a semiconcave function, which has an independent interest. Observe that the nondegeneracy condition (3) in the statement below is always fulfilled in the case of an eikonal equation.

\begin{lemma}\label{l:ew}
Let $u\in SC^\alpha_{loc}(U)$, with $U \subset \R^n$ open. For a given $u_0\in \R$, let $L_0$ be a connected component of the level set $\{ x\in \Omega \ | \ u(x)=u_0\}$. Assume that 
\begin{enumerate}
\item $ \emptyset\not=L_0\subset\subset U$, 
\item $L_0\cap \Sigma (u)=\emptyset$, 
\item $Du(x)\not=0$ for every $x\in L_0$.
\end{enumerate}
Then $L_0$ is a boundaryless compact  $(n-1)$-manifold of class $C^1$.
\end{lemma} 
\begin{proof}
By Theorem \ref{t:generdiff}(i), the superdifferential of $u$ coincides with the Clarke differential, see \cite{Cl}. Hence hypothesis (3) allows to apply the nonsmooth version of the implicit function theorem, see Section 7.1 in \cite{Cl}, to obtain that $L_0$ is locally a Lipschitz graph. 

We further observe that, by Theorem \ref{t:generdiff}(3), $Du$ is continuous on $L_0$. Hence, the restriction of $u$ to $L_0$ satisfies the assumptions of Whitney's extension theorem (see e.g. \cite[Theorem 2.3.6]{H}) and there exists a function $\tilde{u}\in C^1(\R^n)$, such that $\tilde u(x)=u(x) $ and $D\tilde u(x)=Du(x)$ for all $x \in L_0$. The classical implicit function theorem applied to $\tilde u$ now shows that $L_0$ is not only Lipschitz but also a $C^1$ graph. 
\end{proof}

\noindent {\em Proof of Theorem \ref{t:existence}}. 
 For $a>0$, we define  
$$
 \Lambda (a):=\{ x\in X \, :\, d(x)=a\}.
$$
We observe that, since $d(x)\geq c_0|x-k_0|$ in $X$,  $\Lambda (a)$ is compact.

We assume by contradiction that there exists $a_0>d_+$ and a connected component of $\Lambda (a_0)$, which we denote by $K_0$, which does not contain singular points for $d$. By construction, $K_0$ is disjoint from the obstacle $\mathscr O$. Furthermore, we know from \eqref{eq:iconale} that $d$ satisfies in the classical sense 
 $$
\langle A^{-1}(x) D d(x),Dd(x)\rangle =1\qquad \forall x\in K_0.   
 $$
This shows in particular that $Dd(x) \neq 0$ at these points. Hence, from Lemma \ref{l:ew} we deduce that $K_0$ is a boundaryless compact  $(n-1)$-manifold of class $C^1$. Let us now define
$$
T :=
\sup \{ t \in [0,a_0] \ : \ \gamma(t) \notin {\mathscr O} \mbox{ for all } \gamma \in \Gamma^*[x], x \in K_0 \}.
$$
Then $T$ is the first time at which an optimal trajectory starting from $K_0$ touches the obstacle, if such a time exists; if instead all minimizers from $K_0$ reach the target without touching the obstacle, we have $T=a_0$. By definition of $d_+$, we have $T \geq a_0-d_+$.

By the assumption that the points of $K_0$ are regular and from Propositions \ref{l:vi} and \ref{p:nuova} we know that, for any $x \in K_0$, the minimizer for $x$ is uniquely determined for $t \in [0,T]$. Let us denote by $y_x(\cdot)$ such trajectory. We define a map $F: K_0 \times [0,T] \to X$ by setting
$F(x,t)=y_x(t)$ and we set $K_t:=F(K_0,t)$. Proposition \ref{l:vi} also implies that every point of $K_t$ is regular with nonzero gradient, except in the case $t=T=a_0$ when such a set reduces to the point $k_0$.

We observe that $\dot y_x(0)=A^{-1}(x)Dd(x)$ and so it depends continuously on $x$. Therefore the map $F$ is continuous, by the continuous dependence on the initial data of the solution to the geodesic ODE. In addition, $F$ is injective on $K_0 \times [0,T[$ by Proposition \ref{p:nuova}; if $T<a_0$, it is also injective up to $t=T$. Then Brouwer's theorem on the invariance of the domain, see e.g. \cite[Cor. 19.7]{Bredon}, implies that $F$ is an open map on $K_0 \times \, ]0,T[ \,$.


{\bf Claim:} $T<a_0$

\noindent {\it Proof of the claim:}  We argue by contradiction and suppose $T=a_0$. We consider the region swept by the trajectories
$$
 {\mathcal A} := \displaystyle \cup_{0<t<T} K_t = F( K_0 \times ]0,T[),
$$
which is an open subset of $\R^n$, since $F$ is an open map. Observe that
$$
\partial {\mathcal A}=K_0 \cup \{ k_0 \}.
$$
In fact, $K_0 \cup \{ k_0 \}$ is contained in $\partial \mathcal A$ by construction and the assumption $T=a_0$. On the other hand, since $\mathcal A$ is open, any $\bar y \in  \partial {\mathcal A}$ is such that $\bar y \notin {\mathcal A}$ but $\bar y=\lim F(x_k,t_k)$ for some sequences $\{ x_k \} \subset K_0$ and $\{ t_k \} \subset \, ]0,T[$. By compactness, we can assume that $x_k \to \bar x \in K_0$ and $t_k \to \bar t \in [0,T]$, hence $\bar y=F(\bar x,\bar t)$. Since $\bar y \notin {\mathcal A}$, we have that either $\bar t=0$ or $\bar t=T$, which implies $\bar y \in K_0$ or $\bar y=k_0$ respectively.

This implies that ${\mathcal A}$ is a pointed neighbourhood of $k_0$, that is, there exists $\delta>0$ such that
$$
y \in X, \, 0<|y-k_0|<\delta \quad \Longrightarrow \quad y \in \mathcal A.
$$ 
If not, there would be points in $X$ arbitrarily close to $k_0$, different from $k_0$ and not belonging to $\mathcal A$. Then there would also be points $\bar y \in \partial \mathcal A$ different from $k_0$ and arbitrarily close to $k_0$, in contradiction with the property that $\partial {\mathcal A}=K_0 \cup \{ k_0 \}$, where $K_0$ has positive distance from $k_0$.

The above property of $\mathcal A$ easily leads to a contradiction. Consider in fact a point $x_- \in \partial { \mathscr O}$ such that $d(x_-)=\min_{\partial { \mathscr O}} d$. Then any distance minimizing trajectory $\gamma_-$ from $x_-$ to $k_0$ does not touch $ \mathscr O$ except at the initial point. On the other hand, $\gamma_-$ must cross some trajectory $y_x$ starting from $K_0$ because $\mathcal A$ is a pointed neighbourhood of $k_0$. This contradicts the uniqueness property of minimizers away from the obstacle, because $y_x$ is defined in a larger interval $[0,d_+]$ than $\gamma_-$ and does not touch the obstacle. This proves our claim that $T<a_0$.

We now proceed with the proof of our theorem. Since $T<a_0$, there exists $\bar x \in K_0$ such that $\bar y :=y_{\bar x}(T) \in \partial\mathscr O$. We want to show that, similarly to Lemma \ref{l:ew}, $K_T=F(K_0,T)$ is a $C^1$ hypersurface, which therefore touches from outside the obstacle $\mathscr O$ at the point $\bar y$. This time the regularity of $K_T$ is more difficult to analyze because the set is not contained in the interior of $\R^n \setminus \mathscr O$.

For this reason, we consider the semiconcave extension $E(d)$ of $d$ in $B_\delta(\bar y)$ for some $\delta>0$ given by Theorem \ref{t:estensione}. We consider the level set $\{ y \in B_\delta(\bar y) \ : \ E(d)(y)=d(\bar y)=a_0-T\}$, which by construction contains $K_T \cap B_\delta(\bar y)$; such a set may apriori contain other points in the interior of $\mathscr O$ added by the extension. We show that this is not the case and that, at least in a smaller neighbourhood, the level set of the extended function is not larger than the original one. \smallskip

{\bf Claim:} There exists $\delta' \in ]0,\delta]$ such that
$$\{ y \in B_{\delta'}(\bar y) \ : \ E(u)(y)=d(\bar y)\}=K_T \cap B_{\delta'}(\bar y). \smallskip$$ 

\noindent {\it Proof of the claim:} Since $D^*E(d)(\bar y)= D^*d(\bar y)=\{ Dd(\bar y) \}$, with $Dd(\bar y) \neq 0$, we can apply the nonsmooth implicit function theorem and obtain that the level set of $E(d)$ containing $\bar y$ is locally a Lipschitz graph. 
More precisely, let us write $y \in \R^n$ as $y=(y',y_n)$ with $y' \in \R^{n-1}$.  Then, after possibly reordering the coordinates and choosing a smaller $\delta>0$, there is a Lipschitz function $\phi:B_\delta(\bar{y}') \subset \R^{n-1} \to \R$ such that any $y=(y',y_n) \in B_\delta(\bar y)$ satisfies $E(d)(y)=d(\bar y)=a_0-T$ if and only if $y_n=\phi(y')$.

%
%
%
%

Let us write for simplicity $F_T(\cdot)=F(\cdot,T)$. We can choose $\eta>0$ small enough so that $F_T(x) \in B_\delta(\bar y)$ for all  $x \in K_0 \cap B_{\eta}(\bar x)$. We denote by $P:\R^n \to \R^{n-1}$ the projection $P(y',y_n)=y'$. We recall that $F_T$ is injective, and therefore the composition $P \circ F_T$ is continuous and injective from  $K_0 \cap B_{\eta}(x_0)$ to $B_\delta(\bar{y}') \subset \R^{n-1}$. By Brouwer's theorem on the invariance of the domain, $P \circ F_T$ is open. In particular, all points in a neighbourhood of $\bar y'$ of some radius $\delta' \leq \delta$ are projections of points of $F_T(K_0 \cap B_\eta(\bar x) )$.  Since the points of the graph of $\phi$ are in one-to-one correspondence with their projections on the domain of $\phi$, this shows that not only is $F_T(K_0 \cap B_\eta(\bar x) )$  contained in the graph of $\phi$, but it entirely covers the intersection of the graph of $\phi$ with $B_{\delta '}(\bar y)$ as well. That is, any $y \in B_{\delta '}(\bar y)$ such that $E(d)(y)=d (\bar y)$ is also such that $y \in F_T(K_0 \cap B_\eta(\bar x) )\subset K_T$, and the claim is proved.

We can now conclude the proof of our theorem. The last claim shows that, in a neighborhood of $\bar y$, the set $K_T$ coincides with the level set of the semiconcave function $E(d)$ and is a Lipschitz graph. In addition, we know that $d$ is differentiable with nonzero gradient at the points of $K_T$. By using Whitney's extension theorem as in Lemma \ref{l:ew}, we obtain that $K_T \cap B_{\delta'}(\bar y)$ is actually a $C^1$-graph. 

Since $d$ is constant on $K_T$, we have that $D d(\bar y)\in N_{\bar y} \, K_T$, where $N_{\bar y}$ denotes the normal space at $\bar y$. On the other hand $K_T$ and $\partial\mathscr O$ are two $C^1$ hypersurfaces which touch at $\bar y$ without crossing, hence they have the same normal direction at $\bar y$. It follows
that
\begin{equation}\label{eq:ddinNaa}
  Dd(\bar y)\in N_{\bar y} \partial \mathscr{O}.
  \end{equation}

On the other hand, we know that optimal trajectories satisfy at contact points with the obstacle
\begin{equation}\label{eq:vtanaa}
  \dot{y}_{\bar x}(T)=-A^{-1}(\bar y)D d(\bar y)\in T_{\bar y} \partial \mathscr O . 
\end{equation}

Hence \eqref{eq:ddinNaa} and \eqref{eq:vtanaa} yield the contradiction
$$
0= \langle A^{-1}(\bar y)Dd(\bar y), D d(\bar y)\rangle=1, 
$$
since $d$ satisfies the eikonal equation in the classical sense at the points of differentiability. This also holds at boundary points, because at such points we still have $D^*d(\bar y)=\{Dd(\bar y)\}$ and the equality follows by continuity.  
This completes the proof of Theorem \ref{t:existence}.


 \section{An example with no singular boundary points}\label{whale}
 
 We describe an example in the plane where we consider the Euclidean metric. We consider an obstacle  $\mathscr{O} \subset \R^2$ and target $k_0$ as in the picture. 
 
  \begin{figure}
 \begin{tikzpicture}
 \draw[thick, dotted] (0,0) -- (3,0);
\draw [very thick](0,0) arc (-90:90:1cm); 
\draw[thick,dashed] (0,-0.05) arc (-90:90:1.05cm); 
\draw [very thick](0,2) arc (90:180:2cm);
\draw[thick,dashed] (0,2.05) arc (90:180:2.05cm);
\draw [very thick](-2,0) -- (-2,-3);
\draw[thick,dashed] (-2.05,0) -- (-2.05,-3);
\draw [very thick](-2,-3) arc (180:450:5cm and 3cm);
\draw [thick,dotted](8.05,-3) arc (0:90:4.9cm and 3.05cm);
\draw[thick,dashed] (-2.05,-3) arc (180:230:5.05cm and 3.05cm);
\draw[thick,dashed] (-0.3,-5.3) -- (3,-7);
\draw[thick,dotted] (6.25,-5.35) arc (310:360:5.05cm and 3.05cm);
\draw[thick,dotted] (6.25,-5.35) -- (3,-7);

\draw [very thick](0,0) arc (90:360:1.5cm);
\draw [very thick](3,0) arc (90:180:1.5cm);
\draw (2.88,0) arc (90:143:1.4cm);
\draw (0,-3) -- (1.78,-.54);
\filldraw [black] [radius=2pt](0,0)circle ;
\node [] at (0,-0.5) {$P$};
\filldraw [radius=2pt](3,0)circle ;
\node  at (3,-0.5) {$Q$};
\filldraw [radius=2pt](0,-3)circle ;
\node  at (0,-3.5) {$S$};
\node  at (4,-3) {\Large $ \mathscr O$};
\node  at (-3.5,-1) {\large $ \gamma_1$};
\draw [thick,dashed,->] (-3,-1) -- (-2.2,-1);
\node  at (7,1) {\large $ \gamma_2$};
\draw [thick,dotted,->] (7,.5) -- (6,-.5);
\node [] at (2, 2) {$\Sigma(d)$};
\draw [->] (2,1.6) -- (2,.8);
\draw (0.6,0.1) .. controls (2,.5) .. (3,1.5);
\filldraw [black] [radius=2pt](3,-7)circle ;
\node [] at (2.5, -7) {$k_0$};
\end{tikzpicture}
 \end{figure}

The relevant properties of the data are the following:
\begin{itemize}
\item The tangent line to $\partial \mathscr O$ at the point $P$ is also tangent to  $\partial \mathscr O$ at $Q$.
\item The target point $k_0$ is placed in such a way that there are two distinct minimizers from the point $P$ which go around the obstacle on opposite sides. These trajectories are called $\gamma_1$ and $\gamma_2$ and are the dashed and the dotted one in the picture, respectively. Observe that $\gamma_1$ and $\gamma_2$ start from $P$ with the same initial velocity.
\end{itemize}

 We claim that that all points of $\partial \mathscr{O}$ are regular. We begin by observing that $P$ is regular because the only minimizers are $\gamma_1$ and $\gamma_2$ which have the same initial velocity: then by
 Proposition \ref{l:vi} the set $D^*d(P)$ is a singleton and $P$ is regular.
 The other points of $\partial \mathscr{O}$ which lie along the trajectories $\gamma_1$ and $\gamma_2$ are regular, also by Proposition \ref{l:vi}.
 The points of the lower part of $\partial \mathscr O$ which lie neither on $\gamma_1$ nor on $\gamma_2$ can be connected to $k_0$ by a straight segment without crossing the obstacle; hence the segment is the unique minimizer and these points are regular too.
 
It remains to consider the part of the boundary surrounding the ``cavity'' between $P$ and $Q$, like the point $S$ drawn in the picture. In order to reach $k_0$, any minimizer $\gamma$ for $S$ must cross the segment between $P$ and $Q$. Since this segment is a part of $\gamma_2$, $\gamma$ can only cross it at $P$ or at $Q$, because distinct minimizers cannot intersect except on the obstacle. It follows that the initial part of $\gamma$ is a distance minimizing arc from $S$ to either $P$ or $Q$. Observe that any point $S$ in the part of the boundary we are considering can be connected to $P$ by a straight segment. Then if $\gamma$ passes through $P$ it would start with such a segment: but this contradicts property (ii) of Theorem \ref{t:regmin}, since the boundary of the obstacle at $P$ is horizontal and cannot coincide with the direction of the segment from $S$ to $P$. Therefore $\gamma$ passes through $Q$. However, it is easy to see that the minimizing path from $S$ to $Q$ is unique. If $S$ is as in the picture, the minimizer starts with a segment leaving the boundary and hitting it again tangentially at a later point, from which it continues up to $Q$ along the boundary. For points closer to $Q$, the first part may be missing and the minimizer reaches $Q$ by just following the boundary. In both cases, the minimizer is unique and the point is regular. Therefore all boundary points are regular.

As a final remark, it is easy to see that there is an arc of singular points starting from $P$ (but not including $P$, as seen above) which have two minimizing trajectories, reaching the boundary of the obstacle in opposite directions and merging with $\gamma_1$ and $\gamma_2$ respectively. The arc propagates up to infinity and contains in particular the points on high level sets of the distance whose existence follows from Theorem \ref{t:existence}.

  \section{Proof of Theorem \ref{t:propagation}} 
\setcounter{equation}{0}
\setcounter{theorem}{0}

For our next result, it is useful to consider also the distance function from the obstacle, given by
$$
d_{\mathscr{O}}(x):=\min_{y \in \partial \mathscr{O}}\mbox{\rm dist}(x,y), \qquad x\in X.
$$

  \begin{lemma}\label{le1}
For every $x\in \mathring{X}$, $x \neq k_0$, we have 
 \begin{equation}\label{d*d*}
 D^*d(x)\cap  D^*d_{\mathscr{O}}(x)=\emptyset . 
 \end{equation}
 \end{lemma}
  \begin{proof}
  Assume, by contradiction, that there is $p \in D^*d(x) \cap  D^*d_{\mathscr{O}}(x)$ for some $x\in \mathring{X}$, $x \neq k_0$
  . Then Proposition \ref{l:vi} implies that there exists a minimizer $\gamma \in \Gamma^*[x]$ for the distance of $x$ to the target such that $\dot{\gamma}(0)=-A^{-1}(x) p$. The same argument shows that there also exists a minimizer $\zeta$ for the distance to the obstacle $d_{\mathscr{O}}(x)$, parametrized by arclength, such that $\dot{\zeta}(0)=-A^{-1}(x) p$. By definition, $\zeta$  does not touch the obstacle except at the final endpoint. 
 
The arcs $\zeta$ and $\gamma$ are both solutions of the geodesic ODE with the same initial conditions and therefore coincide until they reach the obstacle. But this is a contradiction, since $\zeta$ minimizes the distance from $\mathscr{O}$ and therefore hits $\partial \mathscr{O}$ orthogonally (with respect to the metric induced by $A$), while $\gamma$ can only touch $\partial \mathscr{O}$ tangentially, by Theorem \ref{t:regmin}. This proves the assertion.
\end{proof}
  
  We can now give the proof of Theorem \ref{t:propagation}.
  
 \noindent
\begin{proof} 
Let $x_0\in \Sigma (d)$, with $x_0\notin \{k_0\}\cup \partial \mathscr{O}$.
Then, due to Theorem \ref{t:rb}(i), $d$ is a semiconcave with linear modulus in a suitable open neighborhood of $x_0$ in $X$. We want to apply the abstract propagation result of Theorem \ref{t:acprop} to the present situation.  We recall that
\begin{equation}\label{eq:ellipsoid-a}
D^*d(x_0) \subset \{ p\in \R^n\ | \ \langle A^{-1}(x_0)p,p\rangle =1\}
  \end{equation}
We claim that the inclusion is strict.



For this purpose, let us consider an arc $\zeta$ from $x_0$ to the obstacle attaining the minimum in  $d_{\mathscr{O}}(x_0)$, parametrized by arclength. Then $p_0:=-A(x_0)\dot \zeta(0)$ satisfies $p_0 \in D^*d_{\mathscr{O}}(x_0)$ and 
 $$
\langle A^{-1}(x_0) p_0, p_0 \rangle = \langle \dot \zeta(0), A(x_0)\dot \zeta(0) \rangle =1.
$$
On the other hand, by the previous Lemma, $p_0 \notin D^*d(x_0)$. This shows that the inclusion in \eqref{eq:ellipsoid-a} is strict.

Since $D^+d(x_0)=\mbox{ co} D^*d(x_0)$, this implies that $\partial D^+d(x_0)$ contains some $p$ such that $\langle A^{-1}(x_0)p,p\rangle <1$. This proves that $\partial D^+d(x_0)\setminus D^*d(x_0)\not=\emptyset$.
Then, the existence of a Lipschitz (nontrivial) singular arc follows by Theorem \ref{t:acprop}. The proof of Theorem \ref{t:propagation} is complete. 
\end{proof}

The following example shows that the singular arc, in general, is not necessarily unbounded. 
Indeed, consider the case of the Euclidean distance in $\R^3$ with the obstacle given by a torus. 
 
  \begin{figure}[h!]
 \begin{tikzpicture}
\draw (-3.5,0) .. controls (-3.5,2) and (-1.5,2.5) .. (0,2.5);
\draw[xscale=-1] (-3.5,0) .. controls (-3.5,2) and (-1.5,2.5) .. (0,2.5);
\draw[rotate=180] (-3.5,0) .. controls (-3.5,2) and (-1.5,2.5) .. (0,2.5);
\draw[yscale=-1] (-3.5,0) .. controls (-3.5,2) and (-1.5,2.5) .. (0,2.5);

\draw (-2,.2) .. controls (-1.5,-0.3) and (-1,-0.5) .. (0,-.5) .. controls (1,-0.5) and (1.5,-0.3) .. (2,0.2);

\draw (-1.75,0) .. controls (-1.5,0.3) and (-1,0.5) .. (0,.5) .. controls (1,0.5) and (1.5,0.3) .. (1.75,0);
 
\filldraw [black] [radius=2pt](-5,0)circle ;
 
 \node [] at (-5.5, 0) {$k_0$};
 \filldraw [black] [radius=2pt](0.5,0)circle ;
 \node [] at (1, 0) {$x$};
\end{tikzpicture}
 \end{figure} 
 
 We assume that the torus is symmetric along a horizontal plane which also contains the target $k_0$. Then it is easy to see that all points $x$ lying in the part of the plane surrounded by the torus  belong to $\Sigma(d)$. In fact, from any such point there are two symmetric minimizers going around the obstacle in opposite directions. Thus, we have a disk of singular points, and from any point of the disk the singularity propagates in the horizontal directions. However, the singularities stop when the obstacle is reached and do not further propagate along the boundary of the obstacle. Notice that, at the singular points described above, $d(x)<d_+$. We find again singular points ``behind'' the torus, i.e. on the part of the horizontal plane outside of the torus and opposite to $k_0$, which instead propagate to infinity. We conjecture that this picture is valid in general, and that under the hypotheses of Theorem \ref{t:propagation} the singularity propagates along an unbounded arc traveling towards the infinity. At the present stage, however, this remains an open question.

%
%
%


\section{The Euclidean case in two dimensions}
\setcounter{equation}{0}
\setcounter{theorem}{0}

In this section, we study the singular set of the Euclidean distance function in $\R^2$ in the presence of a connected obstacle. That is, we assume throughout that $n=2$, $A(x) \equiv I$ and that ${\mathscr{O}} \subset \mathcal{O}$ is connected.
 
  
   Let us denote by $S(k_0)$ the ``shadow'' of the obstacle projected from the target point $k_0$, i.e. 
 $$
 S(k_0)=\{ x \in X \,:\, \mbox{ the segment $[k_0,x]$ is not contained in $X$}\}.
 $$
Then any arc connecting a point $x \in S(k_0)$ to $k_0$ must ``go around'' the obstacle. The feature of the two-dimensional case with a connected obstacle is that there are exactly two well-separated sides along which a trajectory can go. Let us state it in a precise form as follows.

\begin{definition}
Let $x_0,x_1 \in S(k_0)$ and let $\gamma_i:[0,T_i] \to X$, with $i=0,1$ be continuous simple curves such that $\gamma_i(0)=x_i$, $\gamma_i(T_i)=k_0$. We say that $\gamma_0$ and $\gamma_1$ travel {\em along the same side of the obstacle} if there is a one-parameter family of curves $\tilde \gamma_s:[0,1] \to X$ depending continuously on $s$, such that $\tilde \gamma_s$ coincides with $\gamma_i$ up to a reparametrization for $s=0,1$ and such that $\tilde \gamma_s(0) \in S(k_0)$ and $\tilde \gamma_s(1)=k_0$ for all $s \in [0,1]$.
\end{definition}
  
It is easy to see that this defines an equivalence relation between simple curves connecting points $x \in S(k_0)$ to $k_0$  which divides them into exactly two classes. Then we can give the following characterization of the singular set. 
  \begin{theorem}\label{tds}
The singular set satisfies  
  \begin{equation}\label{eq:ds}
  \Sigma (d)=\{ x\in X\ |  \# D^*d(x)=2\}. 
  \end{equation}
Moreover, we have 
\begin{equation}\label{eq:dga}
\overline{\Sigma (d)}\subseteq \{x\in X\ | \ \# \Gamma^*[x]>1\}.
\end{equation}
  \end{theorem}
  \begin{proof}

By definition, $x \in \Sigma(d)$ if and only if $\# D^*d(x) \geq 2$. This can only occur at points $x \in S(k_0)$, otherwise the only minimizer is the segment between $x$ and $k_0$ and $D^*d(x)$ is a singleton.
From the properties of minimizers it is easy to see that there cannot be two distinct ones passing along the same side of the obstacle, since there would be a shorter curve passing between them and also avoiding the obstacle. This shows that the minimizers can be at most two, which implies that $\# D^*d(x) = 2$ at any singular point.

%

To prove \eqref{eq:dga} we consider a sequence 
$$
\Sigma (d)\ni x_h\longrightarrow x\in X .
$$
From the previous discussion it follows that, for each $h$, there exist $\gamma_h^1,\gamma_h^2\in \Gamma^* [x_h]$, such that, for each $h$,  all $\gamma_h^1$ pass on one side of the obstacle, while $\gamma_h^2$ pass along the other. Up to a subsequence, we can assume that $\gamma_h^1 \to \gamma_\infty^1$ and $\gamma_h^2 \to \gamma_\infty^2$ in $C^{1,1}$ norm, so that $\gamma_\infty^i \in \Gamma^*[x]$ for $i=1,2$. Since $\gamma_\infty^i$ are limits of curves passing along opposite sides of the obstacle for $i=1,2$, they must be distinct. 
This shows that $\Gamma^*[x]$ contains at least two elements and completes the proof.
\end{proof}

\begin{remark}
The inclusion \eqref{eq:dga} is in general false if we consider a general Riemannian metric
or if $n>2$. We observe also that, by Proposition \ref{l:vi}, the property $ \# \Gamma^*[x]>1$ is weaker than $x \in \Sigma(d)$; however, the result is optimal, as shown by the behaviour of the point $P$ in the example of Section 5.
\end{remark}

Let us also observe that we can strengthen the propagation result of Theorem \ref{t:propagation}
to obtain a two-sided singular arc along which the distance to the obstacle decreases.

\begin{corollary}[Propagation of singularities]\label{cprop}
Let $x\in \Sigma (d)\cap \mathring{X}$. Then, there exists a Lipschitz curve 
$$
\gamma:[-\sigma,\sigma ]\longrightarrow \Sigma (d)
$$
such that  $\gamma (0)=x$ and 
\begin{equation}\label{kp}
d_{\mathscr{O}}(\gamma (-t))>d_{\mathscr{O}}(x)>d_{\mathscr{O}}(\gamma (t)) \quad \text{ for every }t\in ]0,\sigma ].
\end{equation}
 \end{corollary}
\begin{proof}
For a given $x\in \Sigma (d)\cap \mathring{X}$, let $y\in \partial\mathscr{O}$ such that $d_{\mathscr{O}}(x)=|y-x|$.
By Lemma \ref{le1}, we have that 
$$
N:=\frac{x-y}{|x-y|}\notin D^*d(x). 
$$
Moreover, by Theorem \ref{tds}, there exist two distinct unit vectors $p_1, p_2\in D^*d(x)$ such that
$$
D^+d(x)=[p_1,p_2].
$$
The vectors $p_1$ and $p_2$ cannot both lie in the same semicircle between $N$ with $-N$, otherwise the corresponding geodesics, which have initial velocity $p_1$ and $p_2$ respectively, would go around the obstacle on the same side, which contradicts minimality. This implies that $p_2-p_1$ is not parallel to $N$, and so there is a vector $\nu \perp p_2-p_1$ such that 
$$
\langle \nu , x-y\rangle >0 . 
$$

By construction, $\nu$ is in the normal cone to $D^+d(x)$ at all $q \in \,]p_1,p_2[\,$. Since $\,]p_1,p_2[\,= \partial D^+d(x)\setminus D^*d(x)$, Theorem \ref{t:acprop} implies that there exists a singular curve $\gamma:[0,\sigma] \to \Sigma(d)$ such that $\gamma(0)=x$ and
$$
\dot{\gamma}^+(0)=-\nu . 
$$
Therefore, 
$$
\lim_{t\to 0^+} \frac{d_{\mathscr{O}}(\gamma (t))-d_{\mathscr{O}}(x)} t=-\langle \nu , p\rangle ,  
$$
for any $p\in D^+d_{\mathscr{O}}(x)$.  Taking $p=(x-y)/|x-y|$, we obtain  
$$
\lim_{t\to 0^+} \frac{d_{\mathscr{O}}(\gamma (t))-d_{\mathscr{O}}(x)} t<0, 
$$
which proves the second inequality in \eqref{kp} for a sufficiently small $\sigma$. We can repeat the construction with $\nu$ replaced by $-\nu$ and obtain an arc $\tilde \gamma:[0,\sigma] \to \Sigma(d)$ propagating in the opposite direction. By setting $\gamma(t)=\tilde \gamma(-t)$ for $t \in [-\sigma,0]$ we conclude our proof. 
 \end{proof}



\section*{Declarations}
\begin{itemize}
\item {\bf Conflict of interest:} The authors state that there is no conflict of interest. 
\item{\bf Funding:} This work was partly supported by:
\begin{itemize}
\item the Excellence Department Project awarded to the Department of Mathematics, University of Rome Tor Vergata, CUP E83C23000330006;
\item the PRIN 2022 PNRR project, CUP E53D23017910001, funded by the European Union---Next Generation EU;
\item the National Group for Mathematical Analysis, Probability and Applications (GNAMPA) of the Italian Istituto Nazionale di Alta Matematica ``Francesco Severi''. 
\end{itemize}
\end{itemize}


\begin{thebibliography}{}


  
  %
\bibitem{A3}  {\sc P.Albano}, {\it On the local semiconcavity of the solutions of the eikonal equation},  Nonlinear Anal. 73 (2010), no. 2, 458--464. 
  
  


\bibitem{A6} {\sc P.Albano}, {\it On the regularity of the distance  
near the boundary of an obstacle}, J. Math. Anal. Appl. 518 (2023), no. 1, Paper No. 126680, 12 pp.     


\bibitem{ABC1} {\sc P.Albano, V.Basco and P.Cannarsa}, {\it On the extension problem for semiconcave functions with fractional modulus}, Nonlinear Anal. 216 (2022), Paper No. 112669, 12 pp.

\bibitem{ABC2} {\sc P.Albano, V.Basco and P.Cannarsa}, {\it The distance function in the presence of an obstacle}, Calc. Var. Partial Differential Equations 61 (2022), no. 1, Paper No. 13, 26 pp.


\bibitem{AC1} {\sc P.Albano and P.Cannarsa}, {\it Structural properties of singularities of semiconcave functions}, Ann. Scuola Norm. Sup. Pisa Cl. Sci. (4)
 28 (1999), no. 4, 719--740.

\bibitem{AA}{\sc R. Alexander and S. Alexander}, {\it Geodesics in Riemannian manifolds-with-boundary},  Indiana University Mathematics Journal, 30(4), (1981), 481--488.

\bibitem{Al} {\sc S.Alexander},  Distance geometry in Riemannian manifolds-with-boundary. Global differential geometry and global analysis (Berlin, 1979), pp. 12--18, Lecture Notes in Math., 838, Springer, Berlin, 1981.

 \bibitem{ABB} {\sc S.B.Alexander,  I.D.Berg and R.L.Bishop}, {\it  Cut loci, minimizers, and wavefronts in Riemannian manifolds with boundary}, Michigan Math. J. 40 (1993), no. 2, 229--237.

\bibitem{Bredon}{\sc G.E.Bredon}, Topology and geometry, Springer, New York, 1993. 

\bibitem{Ca} {\sc A.Canino}, {\it  On p-convex sets and geodesics},  J. Differential Equations 75 (1988), no. 1, 118--157.

\bibitem{Cl}{\sc F.H.Clarke}, Optimization and nonsmooth analysis, Classics in applied mathematics 5, SIAM, 1990.   


\bibitem{CS} {\sc P.Cannarsa and C.Sinestrari}, Semiconcave functions, Hamilton--Jacobi equations, and optimal control. Birkh\"auser, Boston, 2004.

\bibitem{CDL} {\sc I.Capuzzo Dolcetta and P.L.Lions},  {\it Hamilton--Jacobi equations with state constraints}, 
Trans. Amer. Math. Soc. 318 (1990), no. 2, 643--683.


\bibitem{H}{\sc L.H\"ormander}, The analysis of linear partial differential operators I, Springer 1990.
%
  
\bibitem{L} {\sc P.L.Lions},  Generalized solutions of Hamilton--Jacobi equations. Research Notes in Mathematics, 69. Pitman, Boston, Mass.-London, 1982.

\bibitem{MS} {\sc A.Marino and D. Scolozzi},  {\it Geodesics with obstacles} (Italian) Boll. Un. Mat. Ital. B (6) 2 (1983), no. 1, 1--31.

\bibitem{RS} {\sc J.Rauch and J.Sj\"ostrand}, {\it Propagation of analytic singularities along diffracted rays}, 
Indiana Univ. Math. J. 30 (1981), no. 3, 389--401.



\end{thebibliography}
\end{document}